%% file: jjks.tex
\documentclass[12pt,twoside,reqno]{amsart}

\usepackage[OT1]{fontenc}    %
\usepackage{type1cm}         
\usepackage[english]{babel}  %
\usepackage{amsthm}
\usepackage{wasysym}
\usepackage[dvips]{graphicx}
\usepackage[bf]{caption2}
  
  \setcaptionwidth{.9\textwidth}
\usepackage{psfrag}

\numberwithin{equation}{section}

\theoremstyle{plain}
\newtheorem{theorem}{Theorem}[section]
\newtheorem{corollary}[theorem]{Corollary}

\newtheorem{lemma}[theorem]{Lemma}

\theoremstyle{remark}
\newtheorem{remark}[theorem]{Remark}

\theoremstyle{definition}
\newtheorem{definition}[theorem]{Definition}

\newcommand{\R}{\mathbb{R}}
\newcommand{\Rn}{\R^n}

\newcommand{\N}{\mathbb{N}}

\newcommand{\roo}{\varrho}

\DeclareMathOperator{\por}{por}
\DeclareMathOperator{\dimM}{dim_M}
\DeclareMathOperator{\dimH}{dim_H}
\DeclareMathOperator{\dimp}{dim_p}
\DeclareMathOperator{\dist}{dist}
\DeclareMathOperator{\proj}{proj}

\begin{document}

\title{Asymptotically sharp dimension estimates for $k$-porous sets}

\author[E.\ J\"arvenp\"a\"a \and M.\ J\"arvenp\"a\"a \and A.\
  K\"aenm\"aki \and V.\ Suomala]{Esa J\"arvenp\"a\"a \and Maarit
  J\"arvenp\"a\"a \and Antti K\"aenm\"aki \and \\ Ville Suomala} 

\address{Department of Mathematics and Statistics \\
         P.O. Box 35 (MaD) \\
         FIN-40014 University of Jyv\"askyl\"a \\
         Finland}
\email{esaj@maths.jyu.fi}
\email{amj@maths.jyu.fi}
\email{antakae@maths.jyu.fi}
\email{visuomal@maths.jyu.fi}

\thanks{EJ and MJ acknowledge the support of the Academy of Finland 
(projects \#208637 and \#205806), and VS is indebted to the Finnish
Graduate School of Mathematical Analysis and to the Yrj\"o, Vilho, and Kalle 
V\"ais\"al\"a Fund.}
\date{October 2004}

\begin{abstract}
  In $\R^n$, we establish an asymptotically sharp upper bound for the 
  upper Minkowski dimension of $k$-porous sets having holes of certain size
  near every point in $k$ orthogonal directions at all small scales.
  This bound tends to $n-k$ as $k$-porosity tends to
  its maximum value. 
\end{abstract}

\maketitle

\section{Introduction and notation}

The well-known results on dimensional properties of porous sets $A\subset\R^n$ 
having holes of certain size at all small scales
deal with the Hausdorff dimension, $\dimH$, and the following definition 
of porosity:
\begin{equation}\label{eq:origpor}
\por(A)=\inf_{x\in A}\por(A,x),
\end{equation}
where
\begin{equation}\label{eq:origlocpor}
\por(A,x)=\liminf_{r\downarrow0}\por(A,x,r)
\end{equation}
and
\begin{equation}\label{eq:origscalepor}
\por(A,x,r)=\sup\{\rho:\;\text{there is }z\in\R^n\text{ such that }
B(z,\rho r)\subset B(x,r)\setminus A\}.
\end{equation}
Here $B(x,r)$ is a closed ball with centre at $x$ and radius $r>0$.
Mattila \cite{M1} proved that if $\por(A)$ is close to the maximum value 
$\frac12$, then $\dimH(A)$ cannot be much bigger than $n-1$. 
Salli \cite{sa2}, in turn, verified the corresponding fact for both the upper 
Minkowski dimension of uniformly porous sets and the packing dimension of 
porous sets, and in addition to this, confirmed the
correct asymptotic behaviour for the dimension estimates when porosity tends
to $\frac12$. For other related results on porous sets and measures, 
see \cite{BS},  \cite{EJJ}, \cite {JJ}, \cite{JJ1}, \cite{KR}, \cite{MM}, 
and \cite{MMPZ}.

Clearly, $n-1$ is the best possible upper bound for the dimension of a set
having maximum porosity; any hyperplane serves as an example. Whilst 
a hyperplane has holes of maximum size in one direction which is perpendicular
to the plane,  a $k$-dimensional plane has $n-k$ orthogonal directions with 
maximum holes. Intuitively, it seems natural to expect that the more such 
directions the set has, the smaller its dimension should be.
For examples of Cantor sets, see \cite{sa2} and \cite{KS}.
This leads to the following generalisations of 
\eqref{eq:origpor}--\eqref{eq:origscalepor} introduced in \cite{KS}:

\begin{definition}\label{def:KSpor}
Let $k$ and $n$ be integers with $1\le k\le n$. For all $A\subset\R^n$,
$x \in \R^n$, and $r > 0$, we set 
\begin{align*} 
  \por_k(A,x,r) = \sup\{&\roo : \;\text{there are }
  z_1,\ldots,z_k \in \R^n \text{ such that for every }i\\
 &B(z_i,\roo r)\subset B(x,r)\setminus A \notag  
  \text{ and }(z_i-x)
  \cdot(z_j-x)=0\text{ if } j\neq i\}. \notag
\end{align*}
Here $\cdot$ is the inner product.
The \emph{$k$-porosity} of $A$ at a point $x$ is defined to be
\begin{equation*}
  \por_k(A,x) = \liminf_{r \downarrow 0} \por_k(A,x,r),
\end{equation*}
  and the \emph{$k$-porosity} of $A$ is given by
\begin{equation*}
 \por_k(A) = \inf_{x \in A} \por_k(A,x).
\end{equation*}
\end{definition}

Note that $\por_1(A)=\por(A)$ for all $A\subset\mathbb R^n$. As verified
by K\"aenm\"aki and Suomala in \cite{KS} as a consequence of a conical density
theorem, Definition \ref{def:KSpor} gives necessary tools for extending
Mattila's result to the setting described heuristically above.
Indeed, it turns out that the Hausdorff dimension of any set having
$k$-porosity close to $\frac12$ cannot be much bigger than $n-k$, see \cite[Theorem 3.2]{KS}.
In this paper we generalise this result for the upper Minkowski and
packing dimensions using completely different 
methods. Our main results, Theorem \ref{thm:Mdim} and Corollary
\ref{cor:pdim}, may be viewed as extensions of Salli's results to
$k$-porosity as well. However, in the case $k=1$ the proof we give is somewhat 
simpler than that of Salli's. The dimension estimates we establish are 
asymptotically sharp, see Remark \ref{sharp}.  

We complete this section by introducing the notation we use.  
For integers $0\leq m\leq n$, let $G(n,m)$ be the Grassmann
manifold of all $m$-dimensional linear subspaces of $\Rn$. When
$V\in G(n,m)$, the orthogonal projection onto $V$ is denoted by $\proj_V$. If
$0<\alpha<1$, $V\in G(n,m)$, and $x\in\Rn$, we define
\begin{equation*}
X(x,V,\alpha) = \{y\in\R^n :|\proj_{V^\perp}(y-x)|\leq \alpha|y-x| \},
\end{equation*} 
where $V^\perp\in G(n,n-m)$ is the orthogonal complement of $V$. 
Furthermore, given $V\in G(n,m)$ and $0<\alpha<1$, we say that a set 
$A \subset\Rn$ is \emph{$(V,\alpha)$-planar} if 
\begin{equation*}
A \subset X(x,V,\alpha)
\end{equation*}
for all $x\in A$. The set $A$ is called \emph{$(m,\alpha)$-planar} if it is
$(V,\alpha)$-planar for some $V\in G(n,m)$.  

Let $S^{n-1}$ be the unit sphere in $\mathbb R^n$.
For  the half-spaces we use the notation
\begin{equation*}
H(x,\theta) = \{y\in \R^n : (y-x)\cdot\theta>0\},
\end{equation*}
where $\theta\in S^{n-1}$ and $x\in\Rn$. Moreover, $\partial A$ 
is the boundary of a set $A\subset\mathbb R^n$ and 
$A(r)=\{x\in\Rn : \dist(x,A)\leq r\}$ for all  $r>0$. 

There are many equivalent ways to define the Minkowski dimension of
a given bounded set $A\subset\Rn$, see \cite[\S 5.3]{M}. For us
it is convenient to use the following: Letting $0<\delta<1$ and 
$i\in\N$, we denote by $N(A,\delta,i)$ the minimum number of balls of radius 
$\delta^i$ that are needed to cover $A$. The upper Minkowski
dimension of $A$ is defined by setting
\begin{equation*} 
\dimM(A)=\limsup_{i\to\infty}
\frac{\log N(A,\delta,i)}{\log(\delta^{-i})}. 
\end{equation*}
It is easy to see that this definition does not depend on the choice
of $\delta$. The Hausdorff and packing dimensions, see 
\cite[\S 4.8 and \S 5.9]{M}, are denoted by $\dimH$ and $\dimp$, respectively.

\section{Dimension estimates for $k$-porous sets}

For the purpose of verifying our main results, Theorem \ref{thm:Mdim} and
Corollary \ref{cor:pdim}, we need three technical lemmas. The first one, Lemma
\ref{lemma:hballs}, dealing with $k$-porous sets, follows easily from the 
definitions. The remaining ones, Lemmas \ref{lemma:rec} and 
\ref{lemma:hspaces}, are related to $(m,\alpha)$-planar sets.

For $\sqrt{2}-1<\roo<\tfrac{1}{2}$, we define
\begin{equation}\label{deft}
    t(\roo)=\frac 1{\sqrt{1-2\roo}}
\end{equation}
and
\begin{equation}\label{defrho}
    \delta(\roo)=\frac{1-\roo-\sqrt{\roo^2+2\roo-1}}{\sqrt{1-2\roo}}\,.
\end{equation}
Notice that  
\begin{equation}\label{eq:roo}
  0<\delta(\roo)<4\sqrt{1-2\roo}
\end{equation}
and, in particular, $\delta(\roo)\to 0$ as $\roo\to \tfrac12$. 

The first lemma is a quantitative version of the following simple fact:
Assuming that $\por_k(A,x,R)>\roo$, there exists $z$ such that
$B(z,\roo R)\subset B(x,R)\setminus A$. If $R$ is much larger than $r$, then
$\partial B(z,\roo R)\cap B(x,r)$ is nearly like a piece of a hyperplane. 
Therefore one will not lose much if $A\cap B(x,r)\setminus B(z,\roo R)$
is replaced by $A\cap B(x,r)\setminus H$, where $H$ is a suitable 
half-space. The advantage of this replacement is that $B(x,r)\setminus H$
is convex, whilst $B(x,r)\setminus B(z,\roo R)$ is not.   

\begin{lemma}\label{lemma:hballs}
Given $\sqrt{2}-1<\roo<\tfrac{1}{2}$ and $r_0>0$, assume that $A\subset\Rn$ 
is such that $\por_k(A,x,r)>\roo$ for all $x\in A$ and $0<r<r_0$. 
Then, taking $t=t(\roo)$ as in \eqref{deft}, for any $0<r<\frac{r_0}{2t}$,
$x\in A$, and $y\in A\cap B(x,r)$, there are  orthogonal vectors 
$\theta_1,\dots,\theta_k\in S^{n-1}$ such that for all $i\in\{1,\ldots,k\}$
\begin{equation}\label{empty}
A\cap B(x,r)\cap H(y+2\delta r\theta_i,\theta_i)=\emptyset,
\end{equation}
where $\delta=\delta(\roo)$ is as in \eqref{defrho}.
\end{lemma}

\begin{proof} 
The claim follows directly from Definition \ref{def:KSpor} and 
\cite[Lemma 3.1]{KS}. 
\end{proof}

\begin{lemma}\label{lemma:rec}
For all $0<\alpha<1$ there is a positive integer $M=M(n,\alpha)$ such
that if $C\subset\Rn$ is convex, then $\partial C$ can be
decomposed into $M$ parts all of which are $(n-1,\alpha)$-planar.
\end{lemma}
  
\begin{proof}
Let $C\subset\Rn$ be convex. For any $x\in\partial C$, we may choose
$\theta(x)\in S^{n-1}$ such that $H(x,\theta(x))\cap
C=\emptyset$. This defines a mapping $\theta\colon\partial C\to
S^{n-1}$. Let $\tilde\theta\in S^{n-1}$ and 
$B=B(\tilde\theta,\frac{\alpha}3)\cap S^{n-1}$.
Now, if $x,y\in \theta^{-1}(B)$, 
then $|\theta(y)-\theta(x)|\leq\frac 23\alpha$. 
Since $x\notin H(y,\theta(y))$ and $y\notin H(x,\theta(x))$, this yields to
\begin{equation*}
\dist(y-x,\theta(x)^\perp)=\vert(y-x)\cdot\theta(x)\vert
\leq\tfrac23\alpha|y-x|, 
\end{equation*}
see Figure \ref{angles}, and so 
$y\in X(x,\theta(x)^\perp,\frac23\alpha)$. (Here we use the notation 
$\theta(x)^\perp$ for the orthogonal complement of the line spanned by 
$\theta(x)$.) Combining this with the fact that
$\theta(x)^\perp\subset X(0,\tilde\theta^\perp,\frac\alpha3)$ implies that 
$y\in X(x,\tilde\theta^\perp,\alpha)$, 
and hence $\theta^{-1}(B)$ is $(n-1,\alpha)$-planar. Covering $S^{n-1}$
with $M=M(n,\alpha)$ balls of radius $\frac\alpha3$ and taking their
preimages under $\theta$ gives the claim.   
\end{proof}

\begin{figure}
\psfrag{a}{$x$ }\psfrag{b}{$y$}\psfrag{c}{$\theta(x)$}\psfrag{d}{$\theta(y)$}
\psfrag{e}{$\le\frac 23\alpha$}
\begin{center}
\includegraphics[scale=1.0]{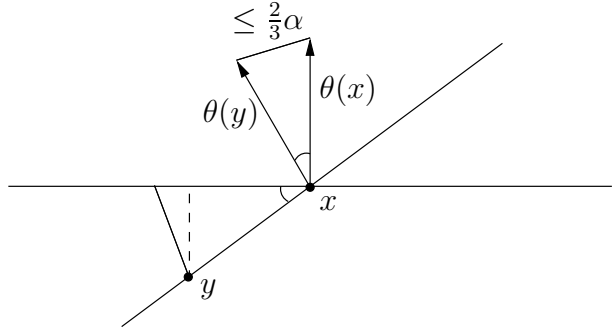}
\end{center}
\caption{Illustration for the proof of Lemma~\ref{lemma:rec}: The extreme 
positions of $x$ and $y$.}
\label{angles}
\end{figure}

The next lemma is used to give a quantitative estimate of how much one needs to translate
a tilted half-space such that it will not intersect a given neighbourhood
of a planar set provided that the untilted half-space does not meet the
neighbourhood.

\begin{lemma}\label{lemma:hspaces}
Letting $0<c<1$, $0<\alpha<\sin(\frac\pi 2-\arccos c)$, and $V\in G(n,m)$,
suppose that $P\subset\Rn$ is $(V,\alpha)$-planar.
If $0<\delta\leq\beta$, $x\in P(\beta)$, $\theta\in S^{n-1}$ with
$|\proj_V(\theta)|\geq c$, and 
$\theta'=\proj_V(\theta)/|\proj_V(\theta)|$, then
\begin{equation*}
H(x+c'\beta\theta',\theta')\cap P(\beta) \subset
H(x+\delta\theta,\theta),
\end{equation*}
where
$c'=c'(\alpha,c)=\frac{2(\sin(\arccos c+\arcsin\alpha))^{-1}+1}
{\sin(\frac\pi 2-\arccos c-\arcsin\alpha)}$.
\end{lemma}

\begin{proof}
We assume that $|\proj_V(\theta)|=c$. In the case $|\proj_V(\theta)|>c$ 
one may use a 
similar argument and show that the number $c'$ can be replaced by a smaller
one. First observe that $P(\beta)\subset X(x,V,\alpha)(2\beta)$. Let
\begin{align*}
A&=X(x,V,\alpha)(2\beta)\setminus H(x+\delta\theta,\theta),\\
w&=x+\delta\theta,\text{ and}\\ 
z&=x-\frac{2\beta\theta}{\sin(\arccos c +\arcsin \alpha)}, 
\end{align*}
and take $y\in A$ which maximises $(y-x)\cdot\theta'$, see
Figure \ref{fig:hspaces}. Now the angle $\varangle wyz$ is
$\frac\pi2-\arccos c-\arcsin\alpha$ and since 
\begin{equation*}
|z-w|\leq\left(\frac 2{\sin(\arccos c+\arcsin\alpha)}+1\right)\beta,
\end{equation*}
we may estimate
\begin{equation*}
\vert(y-x)\cdot\theta'\vert\leq|y-z| 
    =\frac{|z-w|}{\sin(\frac\pi2-\arccos c-\arcsin\alpha)} 
    \leq c'\beta.
\end{equation*}
\end{proof}

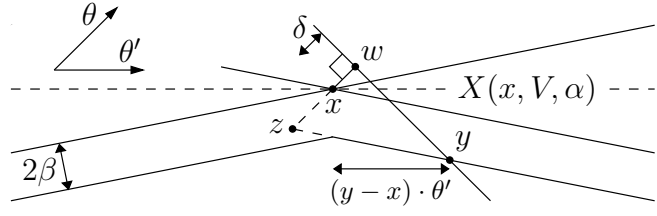
\begin{figure}[t]
  \centering
  \scalebox{1.0}{\input{hspaces.pstex_t}}
  \caption{Illustration for the proof of Lemma \ref{lemma:hspaces}}
  \label{fig:hspaces}
\end{figure}

The following remark will be useful when proving Theorem \ref{thm:Mdim}.

\begin{remark}\label{rem:Lip}
Let $A\subset\Rn$, $0<\alpha<1$, and $V\in G(n,m)$. Then $A$ is
$(V,\alpha)$-planar if and only if there is a Lipschitz mapping 
$f\colon\proj_V(A)\to V^\perp$ (we identify $\Rn$ with the direct sum $V+V^\perp$) 
with Lipschitz constant $\alpha/\sqrt{1-\alpha^2}$ such that $A$ is the
graph of $f$. It follows now from the Kirszbraun's theorem, see
\cite[\S 2.10.43]{Fe}, that $A$ can be extended, that is, there is a
$(V,\alpha)$-planar set $A'\subset\Rn$ such that $A\subset A'$ and 
$\proj_V(A')=V$.  
\end{remark}

Now we are ready to verify our main result concerning the upper Minkowski 
dimension of sets which are uniformly $k$-porous with respect to the scale $r$.

\begin{theorem}\label{thm:Mdim}
Let $0<\roo<\tfrac12$ and $r_0>0$. Assuming that $A\subset\Rn$ is a bounded set
with $\por_k(A,x,r)>\roo$ for every $x\in A$ and $0<r<r_0$, we have
\begin{equation*}
\dimM(A)\leq n-k+\frac c{\log\frac1{1-2\roo}}\,,
\end{equation*} 
where $c=c(n,k)$ is a constant depending only on $n$ and $k$.
\end{theorem}

\begin{proof}
The idea of the proof is as follows: Assuming that all the points in
$A\cap B(x,r)$ are porous and using Lemma \ref{lemma:hballs}, one finds
half-spaces which do not meet $A\cap B(x,r)$. After removing these, one is
left with a convex set $C\subset B(x,r)$ such that all the points in
$A\cap B(x,r)$ are close to the boundary of $C$ and the distance is
proportional to $r\sqrt{\frac12-\roo}$. This implies the claim in the case 
$k=1$.
For $k\ge 2$, we divide the boundary $\partial C$ into planar subsets
$P_i$ and repeat the above process for the projections of each of the sets
$P_i$ into $\mathbb R^{n-1}$. As the result we see that $A\cap B(x,r)$ is
close to a $(n-2)$-dimensional set. This procedure may be repeated $k$
times since there are $k$ orthogonal directions with holes.

Since it is enough to prove the claim for sufficiently large $\roo$, we
may assume that 
\begin{equation}\label{eq:log}
\log\frac1{4\sqrt{1-2\roo}}>\tfrac13\log\frac 1{1-2\roo}.
\end{equation}
Let $0<\alpha<\sin(\frac\pi2-\arccos\frac1{\sqrt{k}})$, and let
$t=t(\roo)$ and $\delta=\delta(\roo)$ be as in \eqref{deft} and
\eqref{defrho}, respectively. For any positive integer $m$ with 
$n-k\leq m\leq n-1$, define $\alpha_m=2^{\frac12(n-k-m+1)}\alpha$. 
Moreover, letting $c_1=c'(\alpha,\frac1{\sqrt{k}})$ be the constant of 
Lemma~\ref{lemma:hspaces}, set $c_2=1+(c_1+1)/\sqrt{1-\alpha^2}$. 

Fix $x\in A$ and $0<r<\frac{r_0}{2t}$. Taking any $y\in A\cap B(x,r)$, let 
$\theta(y)\in S^{n-1}$ be one of the vectors 
$\theta_1,\ldots,\theta_k\in S^{n-1}$ given by Lemma \ref{lemma:hballs}.
Define 
\begin{equation*}
C=\bigcap_{y\in A\cap B(x,r)}\Rn\setminus H\bigl(y+2\delta
r\theta(y),\theta(y)\bigr).  
\end{equation*}
(Here we could replace $\mathbb R^n$ with $B(x,r)$. However, our choice
makes the inductive step somewhat simpler.)
Then $C$ is non-empty and convex, and furthermore by \eqref{empty},
$A\cap B(x,r)\subset(\partial C)(2\delta r)$. 
Using Lemma \ref{lemma:rec}, we obtain 
\begin{equation*}
\partial C=\bigcup_{i=1}^{M(n,\alpha_{n-1})}P_{n-1,i}, 
\end{equation*}
where the constant $M(n,\alpha_{n-1})$ depends only on $n$ and $\alpha_{n-1}$,
and each $P_{n-1,i}$ is $(n-1,\alpha_{n-1})$-planar. This, in turn, gives
that
\begin{equation*} 
A\cap B(x,r)\subset\bigcup_{i=1}^{M(n,\alpha_{n-1})}P_{n-1,i}(2\delta r).
\end{equation*}

\begin{figure}
\psfrag{a}{$z$}\psfrag{b}{$z'$}\psfrag{c}{$V$}\psfrag{d}{$P$}
\psfrag{e}{$\widehat P$}\psfrag{f}{$d_4$}\psfrag{g}{$d_1$}
\psfrag{h}{$\proj_V(z)$}\psfrag{i}{$d_2$}\psfrag{j}{$d_3$}
\begin{center}
\includegraphics[scale=1.0]{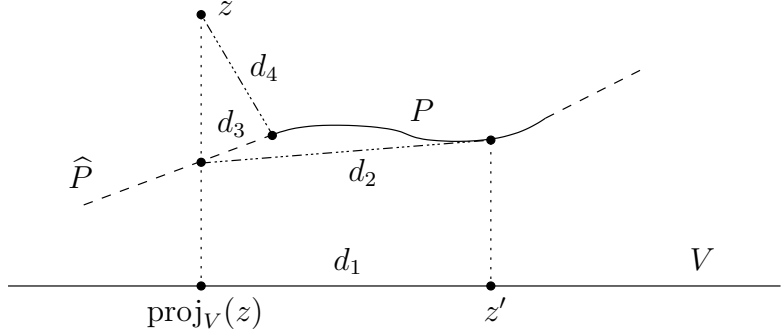}
\end{center}
\caption{A 2-dimensional illustration for the proof of Theorem \ref{thm:Mdim}: 
How much one needs to enlarge the neighbourhood in
the induction step? Here $z\in P(c_2^{n-m-1}2\delta r)$,
$z'\in P'_j(c_1c_2^{n-m-1}2\delta r)$, and $\beta=c_2^{n-m-1}2\delta r$.
Further, $d_1\le c_1\beta$, $d_2\le c_1\beta/\sqrt{1-\alpha^2}$,
$d_3\le\beta/\sqrt{1-\alpha^2}$, and $d_4\le\beta$.}
\label{enlargement}
\end{figure}

If $k\geq 2$, then we continue inductively: Let $n-k<m\le n-1$ and
suppose that we are given  
$(m,\alpha_m)$-planar sets $P_{m,1},\dots,P_{m,l_m}$, where
\[l_m=M(n,\alpha_{n-1})\prod_{j=m+1}^{n-1} M(j,\alpha_j),\] 
such that
\begin{equation*}
A\cap B(x,r)\subset\bigcup_{i=1}^{l_m}P_{m,i}(c_{2}^{n-m-1}2\delta r).
\end{equation*} 
Consider a positive integer $i$ with $1\le i\le l_m$. Abbreviating
$P=P_{m,i}$, let $V\in G(n,m)$ be such that $P$ is
$(V,\alpha_m)$-planar. For every 
$y\in A\cap B(x,r)\cap P(c_{2}^{n-m-1}2\delta r)$, choose orthogonal vectors 
$\theta_1,\ldots,\theta_k\in S^{n-1}$ as in Lemma
\ref{lemma:hballs}. Since $m>n-k$, there is
$\theta\in\{\theta_1,\ldots,\theta_k\}$ for which 
$|\proj_V(\theta)|\geq\frac 1{\sqrt{k}}$. Setting 
$\theta'(y)=\proj_V(\theta)/|\proj_V(\theta)|$, define
\begin{equation*}
C'=\bigcap_{y\in  A\cap B(x,r)\cap P(c_{2}^{n-m-1}2\delta r)}
   V\setminus H\left(\proj_V(y)+c_1 c_{2}^{n-m-1}2\delta r \theta'(y),\theta'(y)\right).
\end{equation*}
It follows from Lemmas \ref{lemma:hballs} and \ref{lemma:hspaces} that 
\begin{equation*}
\proj_V\bigl(A\cap B(x,r)\cap P(c_{2}^{n-m-1}2\delta r)\bigr)\subset 
(\partial C')(c_{1} c_{2}^{n-m-1}2\delta r).\] Moreover,
$C'\subset V$ is convex, and by Lemma \ref{lemma:rec}, its boundary 
$\partial C'$ can be decomposed into $M(m,\alpha_m)$ parts $P'_j$ all of 
which are $(m-1,\alpha_m)$-planar. Using Remark \ref{rem:Lip}, we find
a $(V,\alpha_m)$-planar set $\widehat P$ such that $P\subset\widehat P$ and
$\proj_V(\widehat P)=V$. The r\^ole of $\widehat P$ is to guarantee that
$\widehat P\cap\proj_V^{-1}(P'_j)\ne\emptyset$ for all $j$.  
For all $j\in\{1,\dots,M(m,\alpha_m)\}$ the sets 
$\widetilde{P}_j=\widehat P\cap\proj_V^{-1}(P'_j)$ are 
$(m-1,\alpha_{m-1})$-planar, and moreover, 
\begin{equation*} 
A\cap B(x,r)\cap P(c_{2}^{n-m-1}2\delta r)
\subset\bigcup_{j=1}^{M(m,\alpha_m)} 
\widetilde{P}_j(c_{2}^{n-m}2\delta r),
\end{equation*}
see Figure \ref{enlargement}.
As the result of this inductive process we may find
$(n-k,\alpha_{n-k})$-planar sets $P_{n-k,1},\dots,P_{n-k,l_{n-k}}$, where
$l_{n-k}=M(n,\alpha_{n-1})\prod_{j=n-k+1}^{n-1}M(j,\alpha_{j})$, such that
\begin{equation}\label{final}
A\cap B(x,r)\subset
\bigcup_{i=1}^{l_{n-k}}P_{n-k,i}(c_{2}^{k-1}2\delta r).
\end{equation}

It is not hard to verify that there is a constant $C(\alpha,n,k)$ depending
only on $\alpha$, $n$, and $k$ such that each of the sets
$P_{n-k,i}(c_2^{k-1}2\delta r)\cap B(x,r)$ can be covered with 
$C(\alpha,n,k)\delta^{k-n}$ balls of radius $\delta r$, and therefore 
by \eqref{final}, $C(\alpha,n,k)l_{n-k}\delta^{k-n}$ such balls will cover 
the set $A\cap B(x,r)$. Iterating this and defining 
$c_0=C(\alpha,n,k)l_{n-k}$ gives for all positive integers $i$ that
$N(A,\delta,i)\le N_0(c_0\delta^{k-n})^{i-i_0}$, where $i_0$ is the smallest
integer with $\delta^{i_0}<\frac{r_0}{2t}$ and
$N_0$ is a positive integer such that 
$A\subset\bigcup_{j=1}^{N_0}A\cap B(x_j,\delta^{i_0})$ for some $x_j\in A$. 
Taking logarithms and using \eqref{eq:roo} and \eqref{eq:log} gives 
\begin{equation*}
\begin{split}
  \dimM(A) &\leq
  \limsup_{i\to\infty}\frac{\log\bigl(N_0(c_0\delta^{k-n})^{i-i_0}\bigr)}
  {i\log\frac1\delta}
  =n-k+\frac{\log c_0}{\log\frac1\delta}\\ 
  &\le n-k+\frac{c}{\log\frac1{1-2\roo}}
\end{split}
\end{equation*}
where $c=3\log c_0$ is a constant depending only on $n$ and $k$.
\end{proof}

For the Hausdorff and packing dimensions we have the following immediate
consequence:

\begin{corollary}\label{cor:pdim}
Let $0<\roo<\tfrac12$ and suppose that $A\subset\Rn$ with $\por_k(A)>\roo$. Then 
\begin{equation*}
\dimH(A)\leq\dimp(A)\leq n-k+\frac c{\log\frac1{1-2\roo}},
\end{equation*}
 where $c$ is the constant of Theorem \ref{thm:Mdim}.
\end{corollary}

\begin{proof}
Representing $A$ as a countable union of sets satisfying the assumptions of 
Theorem \ref{thm:Mdim} gives the claim.
\end{proof}

\begin{remark}\label{sharp}  
The estimates of Theorem \ref{thm:Mdim} and Corollary \ref{cor:pdim} are 
asymptotically sharp. In fact, for any $1\le k\le n-1$ there is a constant 
$c'=c'(n,k)$ with the following property: for all $0<\roo<\tfrac12$ there exists
$A_\roo\subset\Rn$ with 
\begin{equation*}
\dim_H(A_\roo)>n-k+\frac{c'}{\log\frac1{1-2\roo}}
\end{equation*} 
and 
$\por_{k}(A_\roo,x,r)>\roo$ for all $x\in\Rn$ and $r>0$. 
The sets $C_\lambda^{k} \times [0,1]^{n-k}$ serve as
natural examples. Here $C_\lambda \subset [0,1]$ is the $\lambda$-Cantor set, 
see \cite[\S 4.10]{M}. When $k=1$, the straightforward calculation can be 
found from Salli \cite[Remark 3.8.2(1)]{sa2}.
\end{remark}

\bibliographystyle{abbrv}
\bibliography{jjks.bib}

\end{document}

%% file: hspaces.pstex_t
\begin{picture}(0,0)%
\includegraphics{hspaces.pstex}%
\end{picture}%
\setlength{\unitlength}{3947sp}%
\begingroup\makeatletter\ifx\SetFigFont\undefined%
\gdef\SetFigFont#1#2#3#4#5{%
  \reset@font\fontsize{#1}{#2pt}%
  \fontfamily{#3}\fontseries{#4}\fontshape{#5}%
  \selectfont}%
\fi\endgroup%
\begin{picture}(4024,1292)(989,-2289)
\put(2980,-2250){\makebox(0,0)[lb]{\smash{\SetFigFont{10}{14.4}{\rmdefault}{\mddefault}{\updefault}{$(y-x)\cdot\theta'$}%
}}}
\put(3777,-1618){\makebox(0,0)[lb]{\smash{\SetFigFont{12}{14.4}{\rmdefault}{\mddefault}{\updefault}{$X(x,V,\alpha)$}%
}}}
\put(1070,-2130){\makebox(0,0)[lb]{\smash{\SetFigFont{12}{14.4}{\rmdefault}{\mddefault}{\updefault}{$2\beta$}%
}}}
\put(1440,-1160){\makebox(0,0)[lb]{\smash{\SetFigFont{12}{14.4}{\rmdefault}{\mddefault}{\updefault}{$\theta$}%
}}}
\put(2750,-1240){\makebox(0,0)[lb]{\smash{\SetFigFont{12}{14.4}{\rmdefault}{\mddefault}{\updefault}{$\delta$}%
}}}
\put(3174,-1376){\makebox(0,0)[lb]{\smash{\SetFigFont{12}{14.4}{\rmdefault}{\mddefault}{\updefault}{$w$}%
}}}
\put(3760,-1911){\makebox(0,0)[lb]{\smash{\SetFigFont{12}{14.4}{\rmdefault}{\mddefault}{\updefault}{$y$}%
}}}
\put(2950,-1710){\makebox(0,0)[lb]{\smash{\SetFigFont{12}{14.4}{\rmdefault}{\mddefault}{\updefault}{$x$}%
}}}
\put(2600,-1835){\makebox(0,0)[lb]{\smash{\SetFigFont{12}{14.4}{\rmdefault}{\mddefault}{\updefault}{$z$}%
}}}
\put(1670,-1390){\makebox(0,0)[lb]{\smash{\SetFigFont{12}{14.4}{\rmdefault}{\mddefault}{\updefault}{$\theta'$}%
}}}
\end{picture}